\documentclass[a4paper,11pt,oneside]{article}
\usepackage[paperwidth=195mm,paperheight=270mm,left=20mm,right=17mm,top=20mm,bottom=20mm,includefoot]{geometry}

\usepackage{cite}
\usepackage{amsmath,amssymb,mathrsfs,amscd,amsthm}
\usepackage{graphicx}
\usepackage{hyperref}
\usepackage{mathtools}
\usepackage{titlesec}

\newtheorem{Theorem}{Theorem}[section]
\theoremstyle{definition}
\newtheorem{Definition}[Theorem]{Definition}
\newtheorem{Lemma}[Theorem]{Lemma}
\newtheorem{Proposition}[Theorem]{Proposition}
\newtheorem{Corollary}[Theorem]{Corollary}
\newtheorem{Remark}[Theorem]{Remark}

\titleformat{\section}
  {\normalfont\Large\centering}   
  {\thesection}
  {1em}
  {}

\titleformat{\subsection}[runin]   
  {\normalfont\large\bfseries}
  {\thesubsection}
  {1em}
  {}
  [.\quad]    

\begin{document}

\title{Cut pairs and Morse splitting of finitely generated groups}

\author{Suzhen Han, Hao Liang, Qing Liu}

\date{ }
\maketitle

\begin{abstract}
Bowditch's theorem for hyperbolic groups establishes a fundamental correspondence between splittings over two-ended subgroups and the existence of local cut points in the Gromov boundary. While analogous results have been obtained for CAT(0) and relatively hyperbolic groups, no general theorem of this type exists for arbitrary finitely generated groups. The Morse boundary $\partial_*\Gamma$, introduced by Charney--Sultan and extended by Cordes, provides a quasi-isometry invariant boundary for any finitely generated group $\Gamma$ that naturally generalizes the Gromov boundary. In this paper, we prove that a splitting of a finitely generated group with connected Morse boundary over a two-ended Morse subgroup gives rise to a separating pair of points in the Morse boundary. 
\end{abstract}

\section{Introduction}

A recurring theme in geometric group theory is that the algebraic splittings of a group are reflected in the topology of its boundary. For Gromov hyperbolic groups, this principle is made precise by Bowditch's celebrated theorem: a one-ended hyperbolic group that is not cocompact Fuchsian splits over a two-ended subgroup if and only if its Gromov boundary contains a local cut point \cite{Bowditch98}. Beyond its intrinsic elegance, this result is foundational because the Gromov boundary is a quasi-isometry invariant; consequently, Bowditch's theorem implies that the JSJ decomposition of a hyperbolic group is preserved under quasi-isometries \cite{Bowditch98}.

Similar phenomena have been investigated for broader classes of groups. Papasoglu and Swenson \cite{PS09} and Haulmark \cite{Haul18} studied (local) cut points (pairs) in the boundaries of $\mathrm{CAT}(0)$ groups and their relationship to splittings over two-ended subgroups \cite{PS09}. 
For relatively hyperbolic groups, the relations between splits between two-ended subgroups and the topology of the boundary were studied by Groff\cite{Groff13}, Guralnik \cite{Gura05}, and Haulmark\cite{Hau19}.

All of these results, however, are confined to groups that are either relatively hyperbolic or CAT(0). Many groups of central interest---such as right-angled Artin groups, mapping class groups, 3-manifold groups---do not belong to the above classes, and thus fall outside the scope of this theory. This motivates the search for a boundary theory that applies to all finitely generated groups and still detects algebraic splittings.

The \emph{Morse boundary} $\partial_* \Gamma$, introduced by Charney and Sultan for $\mathrm{CAT}(0)$ spaces \cite{CS15} and extended by Cordes to all finitely generated groups \cite{Cor17}, provides a natural candidate. It collects all geodesic rays exhibiting hyperbolic-like behavior and is equipped with the direct limit topology. Crucially, the Morse boundary is a quasi-isometry invariant, making it a robust tool for distinguishing groups up to quasi-isometry.  

While the Morse boundary has proven to be a fruitful invariant, its topological properties and their algebraic implications remain poorly understood. Recent work has focused on determining when Morse boundaries are connected or totally disconnected. Fioravanti and Karrer showed that for groups splitting as graphs of groups with undistorted edge groups having empty relative Morse boundary, the connected components of the Morse boundary originate from vertex groups \cite{FK22}. Charney, Cordes, and Sisto classified the Morse boundaries of right-angled Artin groups and showed that the Morse boundaries of cusped hyperbolic \(3\)-manifolds are connected \cite{CCS23}.
Cordes and Levcovitz \cite{CL25} recently showed that connectivity does occur in natural classes of groups. They proved that a one-ended right-angled Coxeter group has connected, non-empty Morse boundary if and only if it is wide-avoidant. Chaika and Hensel \cite{CH24} proved that the Morse boundary of the mapping class group of a surface of genus at least \(2\) is path-connected.

However, a general theorem relating the \emph{local} topology of the Morse boundary to splittings over two-ended subgroups---analogous to Bowditch's theorem---has been absent from the literature.

The purpose of this paper is to establish a partial result in this direction. We prove that for a finitely generated group $\Gamma$ whose Morse boundary $\partial_* \Gamma$ is connected and contains more than two points, a splitting over a two-ended Morse subgroup forces a topological separation in the boundary. Our main theorem is the following.

\begin{Theorem}\label{thm:main}
Let $\Gamma$ be a finitely generated group. Suppose the Morse boundary $\partial_* \Gamma$ is connected and contains more than two points. If $\Gamma$ splits over a two-ended Morse subgroup $\Gamma(e)$, then $\partial_* \Gamma \setminus \Lambda \Gamma(e)$ is disconnected.
\end{Theorem}

\begin{Remark}
The Morse boundaries of Coxeter groups \cite{CL25}, mapping class groups \cite{CH24}, and fundamental groups of cusped hyperbolic \(3\)-manifolds \cite{CCM19} are connected, so Theorem \ref{thm:main} applies to them: any splitting over a two-ended Morse subgroup must cut their Morse boundary.
\end{Remark}

Here $\Lambda \Gamma(e)$ denotes the limit set of $\Gamma(e)$ in the Morse boundary. Since $\Gamma(e)$ is two-ended, its limit set consists of exactly two points. Thus Theorem~\ref{thm:main} asserts that a splitting over a two-ended Morse subgroup gives rise to a separating pair of points in $\partial_* \Gamma$. This can be viewed as the Morse-boundary analogue of one direction of Bowditch's theorem.

The converse of Theorem~\ref{thm:main}---whether the disconnection of $\partial_* \Gamma \setminus \Lambda \Gamma(e)$ implies that $\Gamma$ splits over $\Gamma(e)$---remains open. Resolving this converse would yield a complete characterization of two-ended Morse splittings in terms of the topology of the Morse boundary, and would represent a significant step toward a JSJ theory for general finitely generated groups.

The paper is organized as follows. In Section~\ref{sec:background} we review background on Morse boundaries, Morse subgroups, and limit sets. In Section~\ref{sec:lemmas} we prove the key technical lemmas relating the geometry of the Bass--Serre tree to the topology of the Morse boundary. Section~\ref{sec:proof} contains the proof of Theorem~\ref{thm:main}.

\paragraph{\textbf{Acknowledgements.}}
The authors would like to thank Ruth Charney and Matthew Cordes for the helpful discussions. 
The authors are also grateful to the referee for carefully reading the manuscript and providing numerous helpful comments and suggestions.
Han is supported by the NSFC of China (Grant No. 12401082) and the Fundamental Research Funds for the Central Universities (No. 531118010905).
Liang is partially supported by the NSFC of China (Grant No. 12671085).
Liu is supported by the NSFC of China (Grant No.12301084).
The last author would like to thank the Chern Institute of Mathematics (CIM), where this work was completed during a visit, for the excellent research environment and stimulating academic atmosphere it provided.

\paragraph{\textbf{AI Declaration.}}
AI was only used for typesetting and language clean-up.

\section{Preliminaries}\label{sec:background}

We assume throughout this section that $X$ is a proper geodesic metric space and $o\in X$ is a basepoint.

For a subset \(A\subseteq X\) and a constant \(r\ge 0\), the closed \(r\)-neighbourhood of \(A\) is denoted by

\[
\mathcal{N}_r(A) \coloneqq \{x\in X : d(x,A)\le r\},
\]
where \(d(x,A)=\inf_{a\in A}d(x,a)\). For two subsets \(A_1,A_2\subseteq X\), the \textit{Hausdorff distance} between them is defined by
\[
d_{\mathcal{H}}(A_1,A_2)\coloneqq\inf\{r\ge0:A_1\subseteq\mathcal{N}_r(A_2)\text{ and }A_2\subseteq\mathcal{N}_r(A_1)\}.
\]

\begin{Definition}[Quasi-isometric embedding and quasi-isometry]
Let $K\ge 1$ and $C\ge 0$. A map $f\colon(X,d_X)\to(Y,d_Y)$ between metric spaces is a
\textit{$(K,C)$-quasi-isometric embedding} if for all $x_1,x_2\in X$,
\[
\frac{1}{K}d_X(x_1,x_2)-C\le d_Y(f(x_1),f(x_2))\le Kd_X(x_1,x_2)+C.
\]
If, in addition, $\mathcal{N}_C(f(X))=Y$, then $f$ is called a \textit{$(K,C)$-quasi-isometry}.
\end{Definition}

\begin{Definition}[Morse geodesic]\label{def:morse}
A geodesic $\alpha\subseteq X$ is called \textit{Morse} if there exists a function $N\colon\mathbb{R}_{\ge1}\times\mathbb{R}_{\ge0}\to\mathbb{R}_{\ge0}$ such that every $(K,C)$-quasi-geodesic $\sigma$ whose endpoints lie on $\alpha$ satisfies $\sigma\subset\mathcal{N}_{N(K,C)}(\alpha)$. The function $N$ is called a \textbf{Morse gauge}, and we say that $\alpha$ is \textbf{$N$-Morse}.
\end{Definition}

More generally, a subset $A\subset X$ is called \textbf{$N$-Morse}, if every $(K,C)$-quasi-geodesic $\sigma$ with endpoints lie in $A$ satisfies $\sigma\subset\mathcal{N}_{N(K,C)}(A)$.

If $X$ is a Gromov hyperbolic space, then all geodesics are uniformly Morse. In contrast, the Euclidean plane admits no infinite Morse geodesics whatsoever.
Morse geodesics are fundamental objects for studying non-hyperbolic proper geodesic metric spaces.
Their behavior closely resembles geodesics in hyperbolic spaces, a property formalized by the lemmas below.

\subsection{Morse geodesics and Morse triangles}
We collect below several elementary results concerning Morse geodesics and Morse geodesic triangles that will be used in the later arguments.

\begin{Lemma}\label{subpth}\cite[Lemma~3.1]{Liu21}
For every Morse gauge $N$, there exists a Morse gauge $N_1$, depending only on $N$, such that every subpath of an \(N\)-Morse geodesic is \(N_1\)-Morse.
\end{Lemma}

Recall that a geodesic triangle is said to be \textbf{\(\delta\)-slim} if each of its sides is contained in the \(\delta\)-neighbourhood of the union of the other two sides.
The next lemma, which combines Lemmas~2.3 and~2.4 of \cite{CCM19}, asserts that two Morse sides suffice to guarantee a slim triangle. The original statement for triangles with vertices in \(X\) appeared in \cite{Cor17}.

\begin{Lemma}\label{slim-Morse}\cite{Cor17, CCM19}
    Let $\triangle(p, q, r)$ be a geodesic triangle with vertices in $X\cup\partial_*X$. Assume that two sides of $\triangle(p, q, r)$ are $N$-Morse. Then there exists a positive constant $\delta_N$ and a Morse gauge $N'$, both depending only on $N$, such that the third side of $\triangle(p, q, r)$ is $N'$-Morse and the triangle is $\delta_N$-slim. 
\end{Lemma}
 
  For a bi-infinite Morse geodesic $\alpha$, denote by $\alpha_-,\alpha_+\in\partial_*X$ its endpoints; for a Morse ray $\sigma$, write $\sigma_+$ for its point at infinity.
  From Lemma~\ref{subpth} and Lemma~\ref{slim-Morse}, we immediately obtain the following basic observation.   
  
\begin{Remark}\label{uniform-H-dis}
    Let $\alpha$ be an $N$-Morse bi-infinite geodesic in $X$. Then every geodesic joining the two endpoints $\alpha_-$ and $\alpha_+$ is uniformly Morse, and the Hausdorff distance between any two such geodesics is bounded above by a constant $C_N$ that depends only on $N$.
\end{Remark}

\subsection{Morse boundaries}
We now define the Morse boundary of a proper geodesic metric space \( X \). Fix a basepoint \( o \in X \). 
As a set, the Morse boundary \( \partial_*X_o \) consists of the equivalence classes of all Morse geodesic rays emanating from \( o \), where two rays are equivalent precisely when their Hausdorff distance is finite. To introduce a topology on the Morse boundary, one must take a little care with the class of Morse gauges. We say $N$ is a \textit{refined Morse gauge} if it is non-decreasing in each variable and continuous in the second. For more details, one can check \cite[Lemma~A.4]{CSZ24}.

Given a fixed \textbf{refined} Morse gauge \( N \), define the space
\[
\partial_*^N X_o \coloneqq  \left\{ [\alpha] \mid \text{there exists an } N\text{-Morse geodesic ray } \beta \in [\alpha] \text{ based at } o \right\}.
\]
This space is naturally endowed with the compact-open topology. The space $\partial_*^NX_{o}$ is compact in this topology \cite[Lemma A.5]{CSZ24}.

Let \(\mathcal{M}\) be the family of all refined Morse gauges. For \(N,N' \in \mathcal{M}\), we write the partial order \(N \leq N'\) whenever \(N(K,C) \leq N'(K,C)\) holds for all \(K,C\). The \textit{Morse boundary} is then defined as the direct limit
\[
\partial_* X_o = \varinjlim_{N \in \mathcal{M}} \partial_*^N X_o,
\]
equipped with the direct limit topology; that is, a subset \(U \subseteq \partial_* X_o\) is open if and only if \(U \cap \partial_*^N X_o\) is open for every refined Morse gauge \(N\).

We refer the reader to \cite{Cor17} for more details. For any other basepoint \(o'\in X\), there is a canonical homeomorphism \(\partial_{*}X_{o}\to\partial_{*}X_{o'}\).
Consequently, the Morse boundary is independent of the choice of basepoint, and we simply write \(\partial_* X\) for it.
Any quasi-isometry of proper geodesic metric spaces induces a homeomorphism between their Morse boundaries. Thus the Morse boundary is a quasi-isometry invariant of finitely generated groups; for a group \(G\) we denote its Morse boundary by \(\partial_* G\).

The properties of Morse boundaries and their applications have been investigated in numerous works; see, for instance, \cite{CCM19, CCS23, CS15, Cor17, Liu21, Liu22, CH17, Zb24, HL24, HL26}. Nevertheless, a general theorem relating the local topology of the Morse boundary to splittings over two-ended subgroups has yet to be established. The present paper provides a partial result in this direction.

\section{Limit sets of quasiconvex groups}\label{sec:lemmas}
In this section we study some basic structural properties of limit sets of quasiconvex groups.
Let $\Gamma$ be a finitely generated group acting properly discontinuously on a proper geodesic metric space $(X,d)$ and let $o\in X$ be a basepoint.

\begin{Definition}[Quasiconvex groups]
The group \( \Gamma\) is called \textit{quasiconvex} if there exists a constant \(D \geq 0\) such that for any \(\gamma_1, \gamma_2 \in \Gamma\), every geodesic joining \(\gamma_1 o\) and \(\gamma_2 o\) is contained in the \(D\)-neighbourhood of the orbit \(\Gamma o\).
\end{Definition}

An element $\gamma\in\Gamma$ is called \textbf{Morse} if the orbit map 
\[
n\in \mathbb{Z} \to \gamma^no\in X
\]
is a quasi-isometric embedding and its image is Morse.
From Lemma~4.5 in \cite{Liu21}, such an element has precisely two fixed points on the Morse boundary $\partial_*X$. 
Set $\operatorname{fix}(\gamma) = \{\gamma^+, \gamma^-\}$. This coincides with the limit set (see Definition~\ref{limitset}) of the cyclic subgroup $\langle\gamma\rangle$, and we denote it by $\Lambda(\gamma)$. All Morse geodesics joining $\gamma^-$ and $\gamma^+$ are called \textit{axes} of $g$. By Remark~\ref{uniform-H-dis}, these axes admit a uniform Morse gauge depending only on \(\gamma^-\) and \(\gamma^+\). We say that \(\gamma\) is \textit{\(N\)-Morse} if all its axes are \(N\)-Morse.
 By Theorem~4.6 in \cite{Liu21}, if $\Gamma$ acts geometrically on $X$, then $\gamma$ is a Morse element if and only if the fixed point set $\operatorname{fix}(\gamma)=\{p\in \partial_*X \mid \gamma p=p\}$ is nonempty. 

Define a subgroup $G(\gamma)\leq\Gamma$ as follows:

\[
G(\gamma)\coloneqq \{g\in\Gamma \mid g\operatorname{fix}(\gamma)=\operatorname{fix}(\gamma)\}.
\]

The following proposition describes the structure of \(G(\gamma)\); its proof follows exactly the argument of Lemma~2.11 in \cite{Yang19} for contracting elements and section~6.6 in \cite{DGO17}. It will be used in the next section.

\begin{Proposition} \label{Prop:G(gamma)VirCyc}
    
    Let $\Gamma$ be a finitely generated group acting properly by isometries on $(X, d)$. Fix a basepoint $o\in X$. Let $\gamma\in\Gamma$ be a Morse element. Then the following hold:
    \begin{enumerate}
    \item $\bigl[G(\gamma): \langle \gamma\rangle\bigr] < \infty$.
    \item $G(\gamma) = \bigl\{g\in \Gamma \,\bigm|\, g^{-1}\gamma^{m}g \in \{\gamma^{m}, \gamma^{-m}\} \text{ for some positive integer } m\bigr\}$.
    \item $G(\gamma) = \bigl\{g\in \Gamma \,\bigm|\, d_{\mathcal{H}}\bigl(\langle \gamma\rangle o,\, g\langle \gamma\rangle o\bigr) < \infty\bigr\}$.
\end{enumerate}                
\end{Proposition}

\begin{proof}
    \begin{enumerate}        
        \item Assume that $\gamma$ is $N$-Morse. By the proof of Lemma~4.5 in \cite{Liu21}, there exists a bi-infinite $N$-Morse geodesic $\alpha_{\gamma}$ such that 
        \[
        d_{\mathcal{H}}(\alpha_{\gamma}, \langle \gamma\rangle o)\le C_1,
        \] 
        where the constant $C_1$ depends on $\gamma$ and $o$. For any $g\in G(\gamma)$, Remark~\ref{uniform-H-dis} yields a constant $C_N$, depending only on \(N\), such that 
        \[
        d_{\mathcal{H}}(\alpha_{\gamma}, g(\alpha_{\gamma}))\leq C_N.
        \] 
        Hence, 
        \begin{equation}\label{d_H}
d_{\mathcal{H}}\bigl(\langle \gamma\rangle o,\, g\langle \gamma\rangle o\bigr) \le 2C_1 + C_N.
\end{equation}
        In particular,  $o\in \mathcal{N}_{2C_1+C_N}(g\langle \gamma\rangle o)$,so there exists an integer $m\in\mathbb{Z}$ such that 
        \[
        d(o, g\gamma^{m}o)\le 2C_1+C_N.
        \]
        Define the set $T=\{h\in \Gamma \mid d(o, ho)\le 2C_1+C_N\}\cap G(\gamma)$.. Since \(\Gamma\) acts properly, \(T\) is finite. Because \(g\gamma^m \in T\), we obtain
        \[
        G(\gamma) = \bigcup_{h\in T} h\langle\gamma\rangle,
        \] 
        which proves that the index is finite.
        
        \item If $g\in \Gamma$ satisfies $g^{-1}\gamma^{n}g \in \{\gamma^n, \gamma^{-n}\}$ for some positive integer $n$, then $g\Lambda(\gamma^n)=\Lambda(\gamma^n)=\Lambda(\gamma)$. Hence the right-hand side is contained in $G(\gamma)$.
        
        Conversely, take any $g\in G(\gamma)$. By \eqref{d_H}, for every integer $k$ there exists an integer $n_k$ such that 
        \[
        d(\gamma^{k}o, g\gamma^{n_k}o)\le 2C_1+C_N.
        \]
        Thus $\gamma^{-k}g\gamma^{n_k}\in T$, which is a finite set. Consequently, there exist integers $k< k'$ with
        \[
        \gamma^{-k}g\gamma^{n_k}=\gamma^{-k'}g\gamma^{n_{k'}}.
        \] 
        Setting $m=k'-k>0$ and $j=n_{k'}-n_k$, we obtain
        \[
        \gamma^{m}g=g\gamma^{j}.
        \]
        An easy induction shows that for every positive integer $l$,
        \[
        \gamma^{m^l}g^l = g^l\gamma^{j^l}.
        \]
        Suppose the map $n\to \gamma^{n}o$ is a $(K, C)$-quasi-isometric embedding for some $K\ge 1$ and $C\ge 0$.
        Then
\[
\frac{1}{K}|j|^l - C \le d(\gamma^{j^l}o, o) \le K|j|^l + C,
\quad
\frac{1}{K}m^l - C \le d(\gamma^{m^l}o, o) \le Km^l + C.
\]
Now observe that
\[
d(\gamma^{j^l}o, o) = d\big(g^l\gamma^{j^l}o, g^lo\big) = d\big(\gamma^{m^l}g^{l}o, g^lo\big),
\]
and
\[
\begin{aligned}
\big|d(\gamma^{m^l}o, o) - d(\gamma^{m^l}g^{l}o, g^lo)\big|
&\le d(o, g^lo) + d(\gamma^{m^l}o, \gamma^{m^l}g^{l}o) \\
&= 2d(o, g^lo) \\
&\le 2\sum_{t=1}^{l} d(g^{t-1}o, g^to)
= 2l\,d(o, go).
\end{aligned}
\]
Combining these estimates yields
\[
\frac{1}{K}|j|^l - C \le Km^l + C + 2l\,d(o, go),
\quad
\frac{1}{K}m^l - C \le K|j|^l + C + 2l\,d(o, go)
\]
 for all positive integers \(l\). Since these inequalities hold for arbitrarily large \(l\), we must have \(|j| = m\). Therefore \(g^{-1}\gamma^m g \in \{ \gamma^m, \gamma^{-m} \}\), as required. 
        
        \item For any $g\in \Gamma$, the triangle inequality for the Hausdorff distance gives
        \[
        \bigl|d_{\mathcal{H}}(\alpha_{\gamma}, g\alpha_{\gamma})-d_{\mathcal{H}}(\langle\gamma\rangle o, g\langle\gamma\rangle o)\bigr|\le 2d_{\mathcal{H}}(\alpha_\gamma, \langle\gamma\rangle o)\le 2C_1.
        \]
        Hence $d_{\mathcal{H}}(\alpha_{\gamma}, g\alpha_{\gamma})<\infty$ if and only if $d_{\mathcal{H}}(\langle\gamma\rangle o, g\langle\gamma\rangle o)<\infty$, which completes the proof.
        
    \end{enumerate}
\end{proof}

\begin{Corollary}
    The centralizer of a Morse element in $\Gamma$ is finite-by-cyclic.
\end{Corollary}

The following lemma says that two Morse elements have a common fixed point if and only if they have both fixed points in common.
\begin{Proposition}
 Let $\gamma_1, \gamma_2$ be two Morse elements. Then
\[
\bigl|\Lambda(\gamma_1) \cap \Lambda(\gamma_2)\bigr| \in \{0,2\}.
\]    
\end{Proposition}
\begin{proof} Assume that these two Morse elements have at least one common fixed point;  
    without loss of generality, take $\gamma_1^+=\gamma_2^+$. Fix a basepoint $o\in X$. There exists a constant $C>0$, depending only on \(\gamma_1\) and \(\gamma_2\) and $o$, such that for every sufficiently large integer $n$ one can find a positive integer $m_n$ satisfying
    \[
    d(\gamma_1^{n}o, \gamma_2^{m_n}o)\le C.
    \]
    Since \(\Gamma\) acts properly, there exist distinct positive integers \(k \neq l\) such that
    \[
    \gamma_1^{-k}\gamma_2^{m_k}=\gamma_1^{-l}\gamma_2^{m_{l}}.
    \]
    Set $\gamma=\gamma_1^{k-l}=\gamma_2^{m_k-m_l}$. Then \(\gamma\) is a Morse element and \(\Lambda(\gamma) = \Lambda(\gamma_1) = \Lambda(\gamma_2)\), which implies that the intersection contains exactly two points.
\end{proof}

\subsection{Limit set} 
We say that a sequence $\{x_{n}\} \subseteq X$ converges to a point $p \in \partial_{*} X$, denoted by $\lim_{n\to \infty} x_{n}=p$, if for some Morse gauge $N$ there exist $N$-Morse geodesics $\alpha_n=[o,x_{n}]$ such that every subsequence of $\{\alpha_n\}$ contains a further subsequence that converges uniformly on compact subsets to a geodesic ray $\alpha $ representing $p\in \partial_{*}X$. It is straightforward to verify that this notion of convergence is independent of the choice of basepoint.

\begin{Definition}[Limit set]\label{limitset}
    Let $A\subseteq X$. The \textit{limit set} of $A$ is defined as 
    \[
    \Lambda A \coloneqq \bigl\{p \in \partial_{*} X \bigm| \exists\{x_{n}\} \subseteq A,\lim_{n\to \infty} x_{n}=p \bigr\}.
    \]  
Let $G$ act properly by isometries on a proper geodesic
metric space $X$. The \textit{limit set} of the $G$-action on $\partial_*X$ is defined as $\Lambda(Go)$.
\end{Definition}

The limit set associated to $G$ does not rely on the choice of basepoint,  which allows us to denote this set unambiguously by \(\Lambda G\). 

\begin{Remark}
    By a standard argument, the geodesic \(\alpha\) above is always \(N_1\)-Morse, where \(N_1\) depends only on \(N\). We caution, however, that \(N_1\) need not equal \(N\); examples illustrating this point can be found in \cite{Liu22} and \cite{CSZ24}. If one works instead with a refined Morse gauge $N$, by Lemma~A.5 of \cite{CSZ24}, then \(N_1\) can be taken to be \(N\) itself.
\end{Remark}

We next consider limit sets under intersections. The following lemma states that for quasiconvex subgroups, taking limit sets commutes with intersection. This is well known in hyperbolic spaces; for more general spaces, it was shown for strongly quasiconvex subgroups in \cite{Tran19}.

\begin{Lemma}\label{intersection limit sets}
Let $\Gamma$ be a properly discontinuous group of isometries of the proper geodesic metric space $(X, d)$.
    Let $G$ and $H$ be quasiconvex subgroups of $\Gamma$. Then $\Lambda G\cap \Lambda H=\Lambda(G\cap H).$
\end{Lemma}
\begin{proof}
   It is immediate that $\Lambda(G\cap H)\subset \Lambda G\cap \Lambda H$. 
We now prove the reverse inclusion $\Lambda G\cap \Lambda H\subset\Lambda(G\cap H)$.
Take any $p\in \Lambda G\cap \Lambda H$. By definition, there exists sequences $\{g_n\}\subset G$, $\{h_n\}\subset H$ and a suitable Morse gauge $N$ with the following property: 
the $N$-Morse geodesics $\eta_n=[o, g_no]$ and $\eta_n'=[o, h_no]$ converge uniformly on compact subsets to $N$-Morse geodesic $\eta$ and $\eta'$, respectively, each connecting $o$ and $p$.

Since $G$ is a quasiconvex subgroup, then there exists a constant $L_{G}>0$(depending on $G$ and the basepoint $o$) such that
\[
\eta_n\subset \mathcal{N}_{L_{G}}(Go).\]

For any $t>0$ and $\epsilon=1$, one can find an integer $n_t>0$ so that for all $n\ge n_t$, we have $d(o, g_no)>t$ and $\sup_{s\in [0, t]}d(\eta(s), \eta_n(s))<1$.
It follows that for all $t>0$,
\[
d(\eta(t), Go)< L_G+1.
\]
Similarity, there exists a constant $L_H$ such that for all $t$
\[
d(\eta'(t), Ho)<L_H+1.
\]
By Remark~\ref{uniform-H-dis}, there is a constant $C_N$ depending only on $N$ such that the Hausdorff distance between 
$\eta$ and $\eta'$ is bounded by $C_N$.

For any $t>0$, choose $g_t\in G$ satisfying $d(g_to, \eta(t))<1+L_G$.
There exists $t'=t'(t)>0$ such that $d(\eta(t), \eta'(t'))\le C_N$. 
For this $t'$, pick $h_{t}\in H$ (depending on $t$) with  $d(h_{t}o,\eta'(t'))<1+L_H.$
Set $M=2+L_G+L_H+C_N$. Then for every $t>0$, 
\[
d(g_to, h_to)<M.
\]
Since $\Gamma$ acts properly, there exists some $k\in \Gamma$ such that $h_t^{-1}g_t=k$ for infinitely many values of $t$. We may therefore extract a sequence $t_n\to \infty$ such that 
\[
z_n:=g_{t_{n}}g_{t_1}^{-1}=h_{t_{n}}h_{t_1}^{-1}\in G\cap H.
\]

Since $d(z_no, g_{t_n}o)=d(o, g_{t_1}^{-1}o)$, we have 
$$d(z_no, \eta(t_n))< 1+L_G+d(o, g_{t_1}^{-1}o).$$ Hence all points  $z_no$ lie inside $\mathcal{N}_{D}(\eta)$, where $D=1+L_G+d(o, g_{t_1}o)$.

By the Morse property of $\eta$, we have $[o, z_no]\subset \mathcal{N}_{D'}(\eta)$, where $D'=D'(D, N)$.

Consider the geodesic triangle $\triangle(o, z_no, \eta(t_n))$. By Lemma~\ref{subpth}, the subpath $\eta|_{[0, t_n]}$ is $N_1$-Morse for some Morse gauge $N_1=N_1(N)$. The segment connecting $z_no$ to $\eta(t_n)$ has length bounded by $D$, so it is $N_2$-Morse for some Morse gauge $N_2=N_2(D)$. 
Applying Lemma~\ref{slim-Morse}, the geodesic $[o, z_no]$ is then $N_3$-Morse, with $N_3=N_3(N_1, N_2)$. 

Now take any subsequence of $[o, z_no]$; By the Arzelà-Ascoli theorem and the Morse property, we can extract a further subsequence converging uniformly on compact sets to an $N'$-Morse geodesic ray $\eta_1$, where $N'=N'(N_3)$. Since $[o, z_no]\subset \mathcal{N}_{D'}(\eta)$ one readily sees that
$\eta$ and $\eta_1$ have finite Hausdorff distance. 
This yields $p\in \Lambda(G\cap H)$, completing the proof.

\end{proof}

The following lemma gives a finiteness result: a quasiconvex subgroup has only finitely many conjugates containing a nontrivial power of a given Morse element.

\begin{Lemma}\label{f.m.c.I}
Let $\Gamma$ be a properly discontinuous group of isometries of the proper geodesic metric space $(X, d)$.
Suppose that $G\leq \Gamma$ is quasiconvex, and that $\gamma\in G$ is a Morse element. There exist only finitely many conjugates of $G$, each of which contains a nontrivial power of \(\gamma\).
\end{Lemma}

\begin{proof}
    Suppose that there exists a sequence of elements $\{g_i\}_{i=1}^{\infty}\subset \Gamma$ such that for any $i$ we have $m_i\in \mathbb{N^*}$ satisfying 
    $\gamma^{m_i}\in g_iGg_i^{-1}$. That is, 
    \[
    g_i^{-1}\gamma^{m_i} g_i\in G.
    \]
    Set $\gamma_i=g_i^{-1}\gamma^{m_i} g_i$. Since $\gamma$ is a Morse element, $\gamma_i$ is also a Morse element. 

     We always fix a basepoint $o\in X$. Since $G$ is quasiconvex, there exists a constant $L>0$ depending only on $G$ and $o$ such that any geodesic between $Go$ lies in the $L$-neighborhood of $Go$. 
     For the fixed Morse element $\gamma_i$, from the proof of Lemma~4.5 in \cite{Liu21}, there exists a sequence of geodesics $\alpha_{i, n}$ that converges uniformly on compact sets to an $N$-Morse bi-infinite geodesic $\alpha_i$ for some Morse gauge $N$ depending only on $\gamma$, where the endpoints of $\alpha_{i, n}$ belong to $Go$ and the geodesic $\alpha_i$ is an axis of $\gamma_i$. Thus, from a similar argument as in Lemma~\ref{intersection limit sets}
     \[
     \alpha_i\subset \mathcal{N}_{L+1}(Go).
     \]
 Let $\alpha$ be an axis of $\gamma$.
     Note that $g_{i}(\alpha_i)$ and $\alpha$ share the same points on the Morse boundary. By Remark~\ref{uniform-H-dis}, there exists a constant $C_N$ depending only on $N$ such that
     $d_{\mathcal{H}}(\alpha, g_{i}(\alpha_i))\le C_N$. 
   Let $p=\alpha(0)$. We have that for any $i$,
     \[
    g_i^{-1}(p) \in \mathcal{N}_{L+1+C_N}(Go). 
     \]
     Thus, there exists $g_i'\in G$ such that
     \[
     d(g_i^{-1}(p), g_i'o)\le L+1+C_N.
     \]
     It implies that $d(p, g_ig_i'o)\le L+1+D_N.$ 
     Let $h_i=g_ig_i'$. Since $g_i'\in G$, we have 
     \[
     h_iGh_{i}^{-1}=g_iGg_i^{-1}.
     \]
     By the properness of the action, the ball $B(p, L+1+C_N)$ contains only finitely many group elements. Consequently, the sequence $\{h_i\}$ can take only finitely many distinct elements. Thus, the corresponding conjugate subgroups $\{g_iGg_i^{-1}\}=\{h_iGh_i^{-1}\}$ also have only finitely many distinct possibilities. This completes the proof. 
    
\end{proof}

For boundary dynamics, the algebraic statement of Lemma~\ref{f.m.c.I} is better reformulated topologically. The next proposition says that if a quasiconvex subgroup accumulates at a point in the limit set of a Morse element, then only finitely many conjugates share that accumulation point. This reformulation will be used later.

\begin{Proposition}\label{f.m.c.II}
Let $\Gamma$ be a group acting geometrically by isometries on a proper geodesic metric space $(X, d)$.
    Suppose that $ G \leq \Gamma$ is quasiconvex, and fix a point $p\in \Lambda(\gamma)$ for some Morse element $\gamma\in \Gamma$.
    Then only finitely many conjugates of $G$ in $\Gamma$ contain $p$ in their limit sets.  
\end{Proposition}

\begin{proof}
 We proceed by contradiction. Assume that there exist infinitely many distinct conjugates of $G$ in $\Gamma$, and denote these conjugates by $G_i=g_iGg_i^{-1}$ for some $g_i\in \Gamma$. 
 Without loss of generality, we assume that $p=\gamma^+$, the attracting fixed point of the Morse element \(\gamma\).
   Fix a basepoint $o\in X$. Since $p\in \Lambda G_i$ for each $i$, it follows from an argument analogous to those in Lemma~\ref{intersection limit sets} and Lemma~3.4 in \cite{HL24} that, together with the quasiconvexity of $G_i$ and the fact that \(\gamma\) is a Morse element, there exist a constant \(D_i\), a sequence \(\{g_{i,n}\}_{n=1}^{\infty} \subset G_i\), and a sequence of positive integers \(k_{i,n} \to \infty\) such that
  \[
  d(g_{i,    n}o, \gamma^{k_{i, n}}o)\le D_i.
  \]
  By the properness of the $\Gamma$ action on $X$, we can find two distinct indices $n_1\neq n_2$ satisfying
  \[
  g_{i, n_1}^{-1}\gamma^{k_{i, n_1}}=g_{i, n_2}^{-1}\gamma^{k_{i, n_2}}.
  \]
  Thus, we have a a nontrivial element $g_{i, n_2}g_{i, n_1}^{-1}=\gamma^{k_{i, n_2}-k_{i, n_1}}\in G_i\cap \langle \gamma\rangle.$
 Therefore, each conjugate subgroup \(G_i\) contains a nontrivial power of \(\gamma\). This contradicts Lemma~\ref{f.m.c.I}, which completes the proof.
\end{proof}

\section{Proof of Theorem~\ref{thm:main}}\label{sec:proof}
Let $X$ be a locally finite graph on which $\Gamma$ acts geometrically and let $\Sigma$ be a simplicial tree on which $\Gamma$ acts simplicially. Suppose that the quotient $\Sigma/\Gamma$ is a finite graph and that all edge stabilizers of $\Sigma$ in $\Gamma$ are quasiconvex subgroups. We construct a $\Gamma$-equivariant map $\phi: X\rightarrow \Sigma$ as follow. First, choose a representative vertex $v$ from each \(\Gamma\)-orbit of vertices of $X$, and assign to it an image vertex \(\phi(v)\in\Sigma\); extend this $\Gamma$-equivariantly to all vertices of $X$. Each edge of $X$ is then mapped linearly onto a geodesic segment in $\Sigma$ joining the images of its endpoints. The choices can be arranged so that distinct vertices of \(X\) have disjoint preimages under \(\phi\). Moreover, we may take \(\phi\) to be continuous. Moreover, we may take $\phi$ to be continuous. 

\begin{Lemma}\label{length decreasing}
If $\alpha$ is a geodesic in $X$. Then the length of $\alpha$ is at least the length of $\phi(\alpha)$.
\end{Lemma}

\begin{Lemma}\label{actually moving away}
    Let $\alpha$ be a geodesic in \(X\) connecting vertices $v_1, v_2\in X$. Suppose $\phi(v_1)$ and $\phi(v_2)$ are vertices of $\Sigma$ that are $m$ apart. Then $v_2$ is at least $m$ away from $\phi^{-1}(\phi(v_1))$.    
\end{Lemma}

Let $E(\Sigma)$ and $V(\Sigma)$ be the sets of edges and vertices, respectively. For any $e\in E(\Sigma)$, define $Q(e)=\phi^{-1}(m(e))$, where $m(e)$ is the midpoint of $e$. Note that $Q(e)$ is quasi-isometric to the stabilizer \(\Gamma(e)\) of \(e\) in \(\Gamma\); consequently, \(Q(e)\) is quasiconvex. 

We now show that points of the Morse boundary not lying in the limit set of any vertex stabilizer correspond to directions escaping to infinity in the tree $\Sigma$. The following proposition makes this precise.

\begin{Proposition}\label{embedding}
The set $\partial_*\Gamma\setminus\bigcup_{v\in V(\Sigma)}\Lambda\Gamma(v)$ admits a natural injection into $\partial \Sigma$.  
\end{Proposition}

 \begin{proof}
     Suppose $\alpha$ is an $N$-Morse geodesic ray in $X$ and let $\alpha(0)=o$. Suppose $\phi(\alpha)$ leaves a vertex $v$, travels a distance $t$, and then comes back to $v$. Then $\phi(\alpha)$ goes through an edge $e$ adjacent to $v$ twice. Hence $\alpha$ passes through $Q(e)$ twice. Then by Lemma \ref{actually moving away} the length of this subpath $\alpha'$ of $\alpha$ is at least $2t-2$. Note that $\alpha'$ is in a $D$-neighborhood of $Q(e)$, where $D$ is a uniform constant. 
     On the other hand, by Lemma \ref{actually moving away}, walking along $\alpha'$ actually moves away from $Q(e)$ by a distance at least as big as $t-1$. Hence $t$ is uniformly bounded. 
     Thus, either $\phi(\alpha)$ converges to some ideal point in $\partial \Sigma$, or else remains within a bounded subset of $\Sigma$. In the latter case, $\phi(\alpha)$ passes through some $v\in V(\Sigma)$ infinitely many times. This implies that there exists a sequence $\{g_n\}\subset\Gamma(v)$ and its orbits $\{g_no\}$ lie in a uniform neighborhood $D$ of $\alpha$ and $d(o, g_no)\to \infty$ as $n\to \infty$. By a similar argument in Lemma~\ref{intersection limit sets}, each geodesic $\alpha_n=[o, g_no]$ is $N_1$-Morse for some Morse gauge $N_1=N_1(D, N)$ and $\alpha(\infty)=\lim_{n\to \infty} g_no\in \Lambda\Gamma(v)$. Therefore we obtain  a natural map from $\partial_*\Gamma\setminus\bigcup_{v\in V(\Sigma)}\Lambda\Gamma(v)$ to $\partial \Sigma$. 

     We now show that this map is well-defined. It suffices to show that if two Morse geodesic rays $\alpha_1$ and $\alpha_2$ in $X$ are bounded distance apart, then $\phi(\alpha_1)$ and $\phi(\alpha_2)$ are also bounded distance apart. Note that this follows from Lemma \ref{length decreasing}. 
     
     Finally, we show that the map is injective. Let $\alpha_1$ and $\alpha_2$ be two Morse geodesic rays originating from the same point $o$ such that $\phi(\alpha_1)$ and $\phi(\alpha_2)$ are at a bounded distance apart and they converge to a point $y\in \partial\Sigma$. Then there is a sequence of adjacent vertices $v_1, v_2, \dots$ converging to $y$ such that both $\phi(\alpha_1)$ and $\phi(\alpha_2)$ pass through all these vertices. 
     
     Suppose that $\alpha_1$ and $\alpha_2$ are N-Morse with respect to some Morse gauge $N$. Let $p_i\in\partial_*\Gamma$ denote the boundary point defined by $\alpha_i$ for $i=1, 2$. We argue by contradiction and assume $p_1\neq p_2$. 
     By Lemma~\ref{slim-Morse}, the geodesic $[p_1, p_2]$ is $N'$-Morse, and there exist points $x\in [p_1, p_2], x_1\in \alpha_1$ and $x_2\in \alpha_2$ satisfying
     \[
     d(x, x_1), d(x, x_2)\le \delta_N,
     \] where $N'$ and $\delta_N$ depend only on $N$. 
     
    For each $k$, since $\phi(\alpha_1)$ and $\phi(\alpha_2)$ both pass through $v_k$, each passes through an edge $e_k$ adjacent to $v_k$. Consequently, $\alpha_1$ and $\alpha_2$ both pass through $Q(e_k)$. Let $x_{i,k}$ be the point in $Q(e_k)\cap \alpha_i$ closest to $x$ for $i=1,2$. From Lemma~\ref{subpth}, the geodesics $[x_i, p_i]_{\alpha_i}$ are $N_1$-Morse, while $[x, p_i]$ is $N_2$-Morse for $i=1, 2$; here the Morse gauges satisfy $N_1=N_1(N)$ and $N_2=N_2(N')$. Applying Lemma~\ref{slim-Morse}, each triangle $\triangle(p_i, x_i, x)$ are $\delta_1$-slim, where $\delta_1=\delta_1(N_1, N_2)$. 
    
    Choose $v_k$ so that $d(v_k, \phi(x))$ is sufficiently large to ensure both
\[
x_{i,k} \in [x_i,p_i]_{\alpha_i} \quad \text{and} \quad d(x_{i,k}, [x,p_i]) \le \delta_1 \quad (i=1,2),
\]
    and in particular satisfying $d(v_k,\phi(x)) > C+D$. The constant $C$, which depends only on $N$, will be constructed below. 
    
    Combining Lemma~\ref{subpth} and Lemma~\ref{slim-Morse}, the geodesic segment $[x_{1,k}, x_{2, k}]$ is $N_3$-Morse for some gauge $N_3=N_3(N)$. A standard geometric argument then gives a positive constant $C$ depending only on $N_3$ and $\delta_1$, such that 
    \[
    d(x, [x_{1,k}, x_{2, k}])\le C.
    \] It follows immediately that $C$ depends only on $N$. 
    We now combine the $D$-quasiconvexity of $Q(e_k)$ and the condition that $d(\phi(x), v_k)>C+D$; this produces a contradiction. We therefore conclude $p_1=p_2\in \partial_*X$. 
\end{proof}

The following proposition is a well-known result of Bowditch\cite[Proposition~1.2]{Bowditch98}: in a finite graph of groups, quasiconvexity of edge groups implies quasiconvexity of vertex groups. We include it here without proof. 

\begin{Proposition}\label{v-qc}
   Let \(\Gamma\) be a finitely generated group that splits as a finite graph of groups. If every edge group is quasiconvex in \(\Gamma\), then every vertex group is also quasiconvex.
\end{Proposition}

Proposition~\ref{embedding} provides an injection of a large portion of the Morse boundary into the boundary of the Bass–Serre tree. We are now ready to prove the main theorem. 

\begin{Theorem}[Main Theorem]
Suppose a finitely generated group $\Gamma$ with connected Morse boundary $\partial_*\Gamma$ admits a splitting as a finite graph of groups. For every edge $e$, let the edge group $\Gamma(e)$ be a finite-index subgroup of $G(\gamma_e)$ for some Morse element $\gamma_e$. Then $\partial_*\Gamma\setminus\Lambda\Gamma(e)$ is disconnected.
\end{Theorem}

\begin{proof}
 We fix an edge $e\in E(\Sigma)$. Then its midpoint $m(e)$ splits $\Sigma\cup\partial\Sigma$ into two components $\Sigma_1\cup\partial\Sigma_1$ and $\Sigma_2\cup\partial\Sigma_2$. 
 Let
 \[U_i=\pi^{-1}(\partial\Sigma_i)\cup\bigcup_{v\in V(\Sigma_i)}\Lambda\Gamma(v)\setminus\Lambda\Gamma(e),\]
 where $\pi:\partial_*\Gamma\setminus\bigcup_{v\in V(\Sigma)}\Lambda\Gamma(v)\hookrightarrow\partial\Sigma$ is the natural injective map in Proposition~\ref{embedding}.

Take $v_1\in V(\Sigma_1)$ and $v_2\in V(\Sigma_2)$, then $\Gamma(v_1)\cap\Gamma(v_2)$ fixes the geodesic $\alpha$ between $v_1$ and $v_2$ in $\Sigma$. Note that $e\in \alpha$, so we have $\Gamma(v_1)\cap\Gamma(v_2)\subseteq\Gamma(e)$. By Proposition~\ref{v-qc}, all vertex groups $\Gamma(v)$ are quasiconvex. Thus, by Lemma~\ref{intersection limit sets} $\Lambda\Gamma(v_1)\cap\Lambda\Gamma(v_2)=\Lambda(\Gamma(v_1)\cap\Gamma(v_2))\subseteq\Lambda\Gamma(e)$.

Note that $\Lambda\Gamma(e)=\{\gamma_e^+, \gamma_e^-\}$.
Thus, either $\Lambda(\Gamma(v_1)\cap\Gamma(v_2))=\emptyset$ or there exists a Morse element $\gamma\in \Gamma(v_1)\cap\Gamma(v_2)$ using the similar argument in Proposition~\ref{f.m.c.II}. 
It follows that in the later case $\Lambda\Gamma(v_1)\cap\Lambda\Gamma(v_2)=\Lambda \Gamma(e)$. By Proposition~\ref{f.m.c.II}, there are only finitely many $v\in V(\Sigma)$ such that $\Lambda \Gamma(e)\subset \Lambda\Gamma(v)$.
 Together with Proposition~\ref{embedding}, we get a natural partition  $\partial_*\Gamma\setminus\Lambda\Gamma(e)=U_1\sqcup U_2$.

\medskip
 \noindent\textbf{Claim 1:} the subset $A_i=\phi^{-1}(\Sigma_i)$ is quasi-convex for $i=1,2$.

\begin{proof}[Proof of the Claim 1]
   Suppose that $\alpha\subseteq X$ is a geodesic with endpoints in $A_i$.  Since $\Sigma$ is a tree, any component of $\alpha$ which lies outside $A_i$ has endpoints in $Q(e)$, and thus remains within a bounded distance of $Q(e)$ by the quasi-convexity of $Q(e)$. Considering $Q(e)\subseteq A_i$, we have that $\alpha$ remains a bounded distance from $A_i$. Hence $A_i$ is quasi-convex.
\end{proof}

\medskip
 \noindent\textbf{Claim 2:} the boundary subset $U_i\cup\Lambda\Gamma(e)$ is the limit set of $A_i$, denoted by $\Lambda A_i$.

\begin{proof}[Proof of the Claim 2]
It is obvious that $U_i\cup\Lambda\Gamma(e)\subseteq\Lambda A_i$. On the other hand,
let $\xi\in\Lambda A_i$ and $o\in A_i$. It means that there exists a Morse gauge $N$ and a sequence $\{x_n\}$ in $A_i$ such that geodesics $\gamma_n=[o, x_n]$ is $N$-Morse and $\gamma_n$ converges uniformly on compact sets to a geodesic $\gamma$. Note that the geodesic $\gamma$ is $N_1$-Morse, where $N_1$ depends only on $N$.

Since each $A_i$ is quasiconvex, $\gamma_n$  lies in a uniform neighborhood of $A_i$. An argument analogous to Lemma~\ref{intersection limit sets} implies that $\gamma$ lies in a uniform neighborhood of $A_i$. By reasoning similar to Proposition~\ref{embedding}, $\phi(\gamma)$ either converges to some ideal point in $\partial \Sigma_i$, or passes infinitely many times through some vertex $v\in V(\Sigma_i)$. In the first case, $\pi(\xi)\in \partial\Sigma_i$, so $\xi\in \pi^{-1}(\partial\Sigma_i)$.
In the latter case, $\xi\in\Lambda\Gamma(v_i)$ for some $v_i\in V(\Sigma_i)$. Together, these two cases yield $\xi\in U_i\cup\Lambda\Gamma(e)$. The Claim follows.
\end{proof}

Note that $\partial_*\Gamma=\Lambda\Gamma(e)\sqcup U_1\sqcup U_2$. By \textbf{Claim 2}, each $U_i\cup\Lambda\Gamma(e)$ is closed in $\partial_*\Gamma$. Thus each $U_i$ is open in $\partial_*\Gamma$. Now we only need to verify that $U_i$ is nonempty. Since $\partial_*\Gamma$ contains at least $3$ points, $\Gamma$ is not virtually cyclic, then we can take $g_i\in \bigcup_{v\in V(\Sigma_i)}\Gamma(v)\setminus G(\gamma_e)$ by Proposition \ref{Prop:G(gamma)VirCyc}. Hence $g_i\gamma_e^+,g_i\gamma_e^-\in U_i$ and $U_i$ is nonempty.
Then $\partial_M\Gamma\setminus\Lambda\Gamma(e)$ is disconnected.
\end{proof}

\bibliographystyle{plain}
\bibliography{Ref}

\bigskip
{School of Mathematics, Hunan University, Changsha, Hunan, 410082, P.R.China}

{\tt Email: hansz@hnu.edu.cn}

\bigskip
{
School of Mathematics, Foshan University, Foshan, Guangdong, 528000, P.R.China}

{\tt Email: lianghao1019@hotmail.com}

\bigskip
{School of Mathematical Sciences \& LPMC, Nankai University, Tianjin 300071, P.R.China}

{\tt Email: qingliu@nankai.edu.cn}

\end{document}